\documentclass[a4paper]{amsart}
\usepackage{a4,pstricks,bbm}
\usepackage{amsmath,amssymb,latexsym}
\usepackage{epsf,psfrag,epsfig}
\usepackage[english]{babel}
\usepackage[utf8x]{inputenc}
\usepackage{boxedminipage,scrtime,ulem}

\newtheorem{thm}[subsubsection]{Theorem}
\newtheorem{prop}[subsubsection]{Proposition}
\newtheorem{lemma}[subsubsection]{Lemma}

\newcommand{\SP}{\vspace{0.2cm}\par}

  \renewenvironment{thebibliography}[1]{%
    \begin{oldthebibliography}{#1}%
      \setlength{\parskip}{0ex}%
      \setlength{\itemsep}{0.8ex}%
  }%
  {%
    \end{oldthebibliography}%
  }

\newcommand{\dd}{\mathbf{\mathrm{d}}}

\newcommand{\mS}{\mathbb{S}}

\newcommand{\mmS}{\mathcal{S}}

\newcommand{\nP}{\mathbf{P}}
\newcommand{\nT}{\mathbf{T}}
\newcommand{\nB}{\mathbf{B}}

\newcommand{\mA}{\mathcal{A}}
\newcommand{\mB}{\mathcal{B}}
\newcommand{\mE}{\mathcal{E}}
\newcommand{\mR}{\mathcal{R}}

\newcommand{\Seq}[1]{\mathrm{Seq}\left(#1\right)}

\newcommand{\bA}{\textrm{A}}

\newcommand{\gA}{\textbf{A}}
\newcommand{\gB}{\textbf{B}}

\begin{document}

\author[J. Ru\'e]{Juanjo Ru\'e}
\address{J. Ru\'e: CNRS, Laboratoire d'Informatique, \'Ecole Polytechnique, 91128 Palaiseau Cedex, France}
\email{rue1982@lix.polytechnique.fr}

\author[I. Sau]{Ignasi Sau}
\address{I. Sau: CNRS, LIRMM, Montpellier, France}
\email{ignasi.sau@lirmm.fr}

\author[D. M. Thilikos]{Dimitrios M. Thilikos}
\address{D. M. Thilikos: Department of Mathematics, National and Kapodistrian University of Athens, Greece}
\email{sedthilk@math.uoa.gr}


\title[Asymptotic enumeration of non-crossing partitions on surfaces]{Asymptotic enumeration of non-crossing\\partitions on surfaces$^{\star}$}

\thanks{$^{\star}$Most of the results of this paper were announced in the extended abstract ``\emph{Dynamic programming for graphs on surfaces. Proc. of ICALP'2010, volume 6198 of LNCS, pages 372-383}'', which is a combination of an algorithmic framework (whose full version can be found
in~\cite{RST10_algo_Arxiv}) and the enumerative results presented in this paper.\\
The first author is supported by the European Research Council under the
European Community's 7th Framework Programme, ERC grant agreement no
208471 - ExploreMaps project. The second author is supported by projects
ANR Agape and ANR Gratos.
The third author is supported by the project ``Kapodistrias'' (A${\rm
\Pi}$ 02839/28.07.2008) of the National and Kapodistrian University of
Athens.}

\date{}

\maketitle
\begin{abstract}\noindent
We generalize the notion of non-crossing partition on a disk to general
surfaces with boundary. For this, we consider a surface $\Sigma$
and introduce the number  $C_{\Sigma}(n)$ of non-crossing partitions of a
set of  $n$ points laying  on the boundary of $\Sigma$.
 Our proofs use bijective techniques arising from map enumeration, joint with the symbolic method and singularity analysis on generating functions.
An outcome of our results is that the exponential growth of
$C_{\Sigma}(n)$ is the same as the one of the $n$-th Catalan number, i.e.,
does not change when we move from the case where $\Sigma$ is a disk to
general surfaces with boundary.

 \end{abstract}


\section{Introduction}
\label{sec:intro}

In combinatorics, a \emph{non-crossing partition} of size $n$ is a
partition of the set $\{1,2,\dots ,n\}$ with the following property: if
$1\leq a<b<c<d\leq n$ and a subset of the non-crossing partition contains
$a$ and $c$, then no other subset contains both $b$ and $d$. One can
represent such a partition on a disk by placing $n$ points on the boundary
of the disk, labeled in cyclic order, and drawing each subset as a convex
polygon (also called \emph{block}) on the points belonging to the subset.
Then, the ``non-crossing'' condition is equivalent to the fact that the
drawing is plane and the blocks are pairwise disjoint. The enumeration of
non-crossing partitions of size $n$ is one of the first nontrivial
problems in enumerative combinatorics: it is well-known that the number of
these structures (either by using direct root
decompositions~\cite{FlajoletNoy:Noncrossing} or bijective
arguments~\cite{Stanley2}) corresponds to
Catalan numbers. More concretely, the number of non-crossing partitions of $\{1,2,\dots,n\}$ on a disk is equal to the Catalan number $C(n)=\frac{1}{n+1}\binom{2n}{n}$. This paper deals with the generalization of the notion of non-crossing partition on surfaces of higher genus with boundary, orientable or not. \\

\paragraph{\bf Non-crossing partitions on surfaces.}
Let $\Sigma$ be a surface with boundary, assuming that this boundary is a
collection of cycles.
 Let also $S$ be a set of $n$ points on the boundary of $\Sigma$
(we assume that $\Sigma$ is a closed set). A {\em partition $\mathcal{P}$}
of $S$ is {\em non-crossing on} $\Sigma$ if there exists a collection
$\mathcal{S}=\{X_{1},\ldots,X_{r}\}$ of mutually non-intersecting
connected closed subsets of $\Sigma$ such that $\mathcal{P}=\{X_{1}\cap
S,\ldots,X_{r}\cap S\}$. We define by $\Pi_{\Sigma}(n)$ the set of all
non-crossing partitions of $\{1,\ldots,n\}$ on $\Sigma$ and we denote
$C_{\Sigma}(n)=|{\Pi}_{\Sigma}(n)|$.

In the elementary case where $\Sigma$ is a disk, the enumeration of
non-crossing partitions can be directly reduced by bijective arguments to
the map enumeration framework and therefore, in this case $C_{\Sigma}(n)$
is the $n$-th Catalan number. However, to  generalize the notion of
non-crossing partition to surfaces of higher genus is not straightforward.
The main difficulty is that there is not a bijection between non-crossing
partitions of a set of size $n$ on a surface $\Sigma$ and its geometric
representation (see Figure~\ref{twodif} for an example of a partition with
two different geometric representations).
%
\begin{figure}[htb]
\begin{center}
\includegraphics[scale=.2]{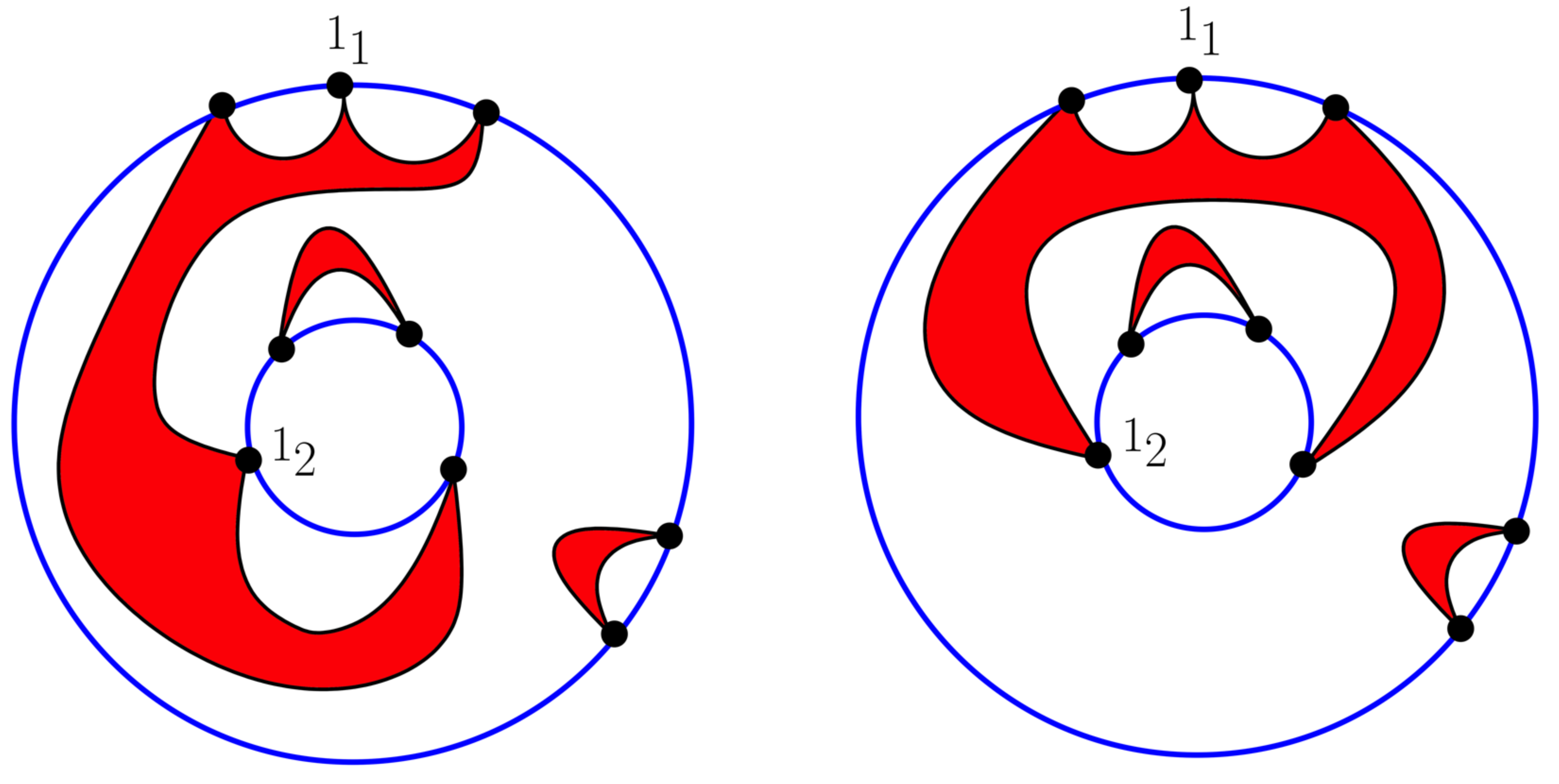}
\caption{Two different representations of the same
partition.}\label{fig:intro-surfaces}
\end{center}
\label{twodif}
\end{figure}
In this paper we study enumerative properties of this geometric
representation. From this study we deduce asymptotic estimates for the
subjacent non-crossing partitions for every surface $\Sigma$.

\SP
\paragraph{\textbf{Our results and techniques.}} The main result of this paper is the following: let $\Sigma$ be a surface with Euler characteristic
$\chi\left(\Sigma\right)$ and whose boundary has
$\beta\left(\Sigma\right)$ connected components. Then the number of
non-crossing partitions on $\Sigma$, $C_{\Sigma}(n)=|\Pi_{\Sigma}(n)|$,
verifies the asymptotic upper bound
\begin{equation}\label{eq:main}
|\Pi_{\Sigma}(n)|\leq _{n\rightarrow
\infty}\frac{c(\Sigma)}{\Gamma\left(-3/2\chi(\Sigma)+\beta(\Sigma)\right)}\cdot
n^{-3/2\chi(\Sigma)+\beta(\Sigma)-1} \cdot 4^n,
\end{equation}
where $\Gamma$ is the Gamma function:
$\Gamma(u)=\int_{0}^{\infty}t^{u-1}e^{-t}dt$. (For a bound on $c(\Sigma)$,
see Section~\ref{apen:bounds-C}.) This upper bound, together with the fact
that every non-crossing partition on a disk admits a realization on
$\Sigma$ (in other words, $C(n)\leq C_\Sigma(n)$), give the result
\begin{eqnarray}
\lim_{n\rightarrow \infty} C_{\Sigma}(n)^{1/n}=\lim_{n\rightarrow
\infty}C(n)^{1/n}=4.
\end{eqnarray}
In other words, $C_{\Sigma}(n)$ has the {\sl same} exponential growth as
the Catalan numbers, no matter the surface $\Sigma$.

In order to get the upper bound~\eqref{eq:main}, we argue in three levels:
we start from a topological level, stating the precise definitions of the
objects we want to study, and showing that we can restrict ourselves to
the study of hypermaps and bipartite maps~\cite{cori}. Once we restrict
ourselves to the map enumeration framework, we use the ideas
of~\cite{BernardiRue:triangulations-higher-genus}, joint with the work by
Chapuy, Marcus, and Schaeffer on the enumeration of higher genus
maps~\cite{schaeffer-highergenus} and
constellations~\cite{chapuy-constelations} in order to obtain
combinatorial decompositions of the dual maps of the objects under study.
Finally, once we have explicit expressions for the generating functions of
these combinatorial families, we study generating functions (formal power
series) as analytic objects. In the analytic step, we extract singular
expansions of the counting series from the resulting generating functions.
We derive asymptotic formulas from these singular expansions by extracting
coefficients, using the Transfer Theorems of singularity
analysis~\cite{FlajoletSedgewig:analytic-combinatorics, FlaOdl}.

\SP
\paragraph{\textbf{Application to algorithmic graph theory.}} The asymptotic analysis carried out
in this paper has important consequences in the design of algorithms for
graphs on surfaces: the enumeration of non-crossing partitions has been
used in~\cite{RST10_algo_Arxiv} to build a framework for the design of
$2^{O(k)}\cdot n^{O(1)}$ step dynamic programming algorithms to solve a
broad class of \textsc{NP}-hard optimization problems for surface-embedded
graphs on $n$ vertices of branchwidth at most $k$. The approach is based
on a new type of branch decomposition called \emph{surface cut
decomposition}, which generalizes sphere cut decompositions for planar
graphs introduced by Seymour and Thomas~\cite{SeymourT94}, and where
dynamic programming should be applied for each particular problem. More
precisely, the use of surface cut decompositions yields algorithms with
running times with a {\em single-exponential dependence} on branchwidth,
and allows to unify and improve all previous results in this active field
of parameterized complexity~\cite{DFT06,DPBF05}. The key idea is that the
size of the tables of a dynamic programming algorithm over a surface cut
decomposition can be upper-bounded in terms of the non-crossing partitions
on surfaces with boundary. See~\cite{RST10_algo_Arxiv} for more details
and references.

\SP
\paragraph{\textbf{Outline of the paper.}} In Section~\ref{sec:prelim} we include all the definitions
and the required background concerning topological surfaces, maps on
surfaces, the symbolic method in combinatorics, and the singularity
analysis on generating functions. In Section~\ref{sec:enumeration} we
state the precise definition of non-crossing partition on a general
surface, as well as the connection with the map enumeration framework.
Upper bounds for the number of non-crossing partitions on a surface
$\Sigma$ with boundary are obtained in Section~\ref{sec:upper-bounds}, and
the main result is proved. A more detailed study of the constant
$c(\Sigma)$ of Equation~\eqref{eq:main} is done in
Section~\ref{apen:bounds-C}.

%

\section{Background and definitions}\label{sec:prelim}
In this section we state all the necessary definitions and results needed
in the sequel. In Subsection~\ref{sec:prelim:topo} we state the main
results concerning topological surfaces, and in
Subsection~\ref{sec:prelim:embedded} we recall the basic definitions about
maps on surfaces. Finally, in Subsection~\ref{sec:prelim:symbolic} we make
a brief summary of the symbolic method in combinatorics, as well as the
basic techniques in singularity analysis on generating functions.

\subsection{Topological surfaces}\label{sec:prelim:topo}
In this work, surfaces are compact (hence closed and bounded) and their
boundary is homeomorphic to a finite set
(possibly empty) of disjoint simple circles. We denote by
$\beta\left(\Sigma\right)$ the number of connected components of the
boundary of a surface $\Sigma$. The Surface Classification
Theorem~\cite{Mohar:graphs-on-surfaces} asserts that a compact and
connected surface without boundary is determined, up to homeomorphism, by
its Euler characteristic $\chi\left(\Sigma\right)$ and by its
orientability. More precisely, orientable surfaces are obtained by adding
$g\geq 0$ \emph{handles} to the sphere $\mS^2$, obtaining a surface with
Euler characteristic  $2-2g$. Non-orientable surfaces are obtained by
adding $h>0$ \emph{cross-caps} to the sphere, getting a non-orientable
surface with Euler characteristic $2-h$. We denote by $\overline{\Sigma}$
the surface (without boundary) obtained from $\Sigma$ by gluing a disk on
each of the $\beta(\Sigma)$ components of the boundary of $\Sigma$. It is
then easy to show that $\chi\left(\overline{\Sigma}\right)=\beta(\Sigma) +
\chi(\Sigma)$. In other words, surfaces under study are determined, up to
homeomorphism, by their orientability, their Euler characteristic, and t
he number of connected components of their boundary. %

A \emph{cycle} on $\Sigma$ is a topological subspace of $\Sigma$ which is
homeomorphic to a circle. We say that a cycle $\mS^1$ \emph{separates}
$\Sigma$ if $\Sigma\setminus\mS^1$ has two connected components. The
following result concerning a separating cycle is an immediate consequence
of~\cite[Proposition~4.2.1]{Mohar:graphs-on-surfaces}.
\begin{lemma}\label{lemma:particio}
Let $\Sigma$ be a surface with boundary and let $\mathbb{S}^1$ be a
separating cycle on $\Sigma$. Let $V_1$ and $V_2$ be connected surfaces
obtained by cutting $\Sigma$ along $\mathbb{S}^1$ and gluing a disk on the
newly created boundaries. Then $\chi(\Sigma)=\chi(V_1)+\chi(V_2)-2$.
\end{lemma}

\subsection{Maps on surfaces and duality}\label{subs:maps-duality}

\label{sec:prelim:embedded} Our main reference for maps is the monograph
of Lando and Zvonkin~\cite{rusos:maps}. A \emph{map} on $\Sigma$ is a
partition
of $\Sigma$ in zero, one, and two dimensional sets homeomorphic to zero,
one and two dimensional open disks, respectively (in this order,
\emph{vertices}, \emph{edges}, and \emph{faces}). The set of vertices,
edges and faces of a map $M$ is denoted by $V(M)$, $E(M)$, and $F(M)$,
respectively. We use $v(M),e(M)$, and $f(M)$ to denote $|V(M)|, |E(M)|$,
and $|F(M)|$, respectively. The \emph{degree} $\dd(v)$ of a vertex $v$ is
the number of edges incident with $v$, counted with multiplicity (loops
are counted twice). An edge of a map has two ends (also called
\emph{half-edges}), and either one or two sides, depending on the number
of faces which is incident with. A map is \emph{rooted} if an edge and one
of its half-edges and sides are distinguished as the root-edge, root-end,
and root-side, respectively. Observe that rooting on orientable surfaces
usually omits the choice of a root-side because the subjacent surface
carries a global orientation, and maps are considered up to
orientation-preserving homeomorphism. Our choice of a root-side is
equivalent in the orientable case to the choice of an orientation of the
surface. The root-end and -sides define the root-vertex and -face,
respectively. Rooted maps are considered up to cell-preserving
homeomorphisms preserving the root-edge, -end, and -side. In figures, the
root-edge is indicated as an oriented edge pointing away from the root-end
and crossed by an arrow pointing towards the root-side (this last,
provides the orientation in the surface). For a map $M$, the \emph{Euler
characteristic} of $M$,  which is denoted by $\chi(M)$, is the Euler
characteristic of the underlying surface.
\\
\paragraph{\textbf{Duality}.} Given a map $M$ on a surface $\Sigma$ without boundary, the \emph{dual map} of $M$, which we denote by $M^*$, is a map on $\Sigma$ obtained by drawing a vertex of $M$ in each face of $M$ and an edge of $M$ across each edge of $M$. If the map $M$ is rooted, the root-edge $e$ of $M$ is defined in the natural way: the root-end and root-side of $M$ correspond to the side and end of $e$ which are not the root-side and root-end of $M$, respectively. This construction can be generalized to surfaces with boundary in the following way: for a map $M$ on a surface $\Sigma$ with boundary, notice that the (rooted) map $M$ defines a (rooted) map $\overline{M}$ on $\overline{\Sigma}$ by gluing a disk (which becomes a face of $\overline{M}$) along each boundary component of $\Sigma$. We call these faces of $\overline{M}$ \emph{external}. Then the usual construction for the dual map $\overline{M}^{*}$ applies using
the external faces. The dual of a map $M$ on a surface $\Sigma$ with
boundary is the map on $\overline{\Sigma}$, denoted $M^{*}$, constructed
from $\overline{M}^{*}$ by splitting each external vertex of
$\overline{M}^{*}$. The new vertices that are obtained are called
\emph{dangling leaves}, which have degree one. Observe that we can
reconstruct the map $M$ from $M^{*}$, by pasting the dangling leaves
incident with the same face, and applying duality. An example of this
construction is shown in Figure~\ref{fig:dual-map}.

\begin{figure}[htb]
\begin{center}
\includegraphics[scale=.257]{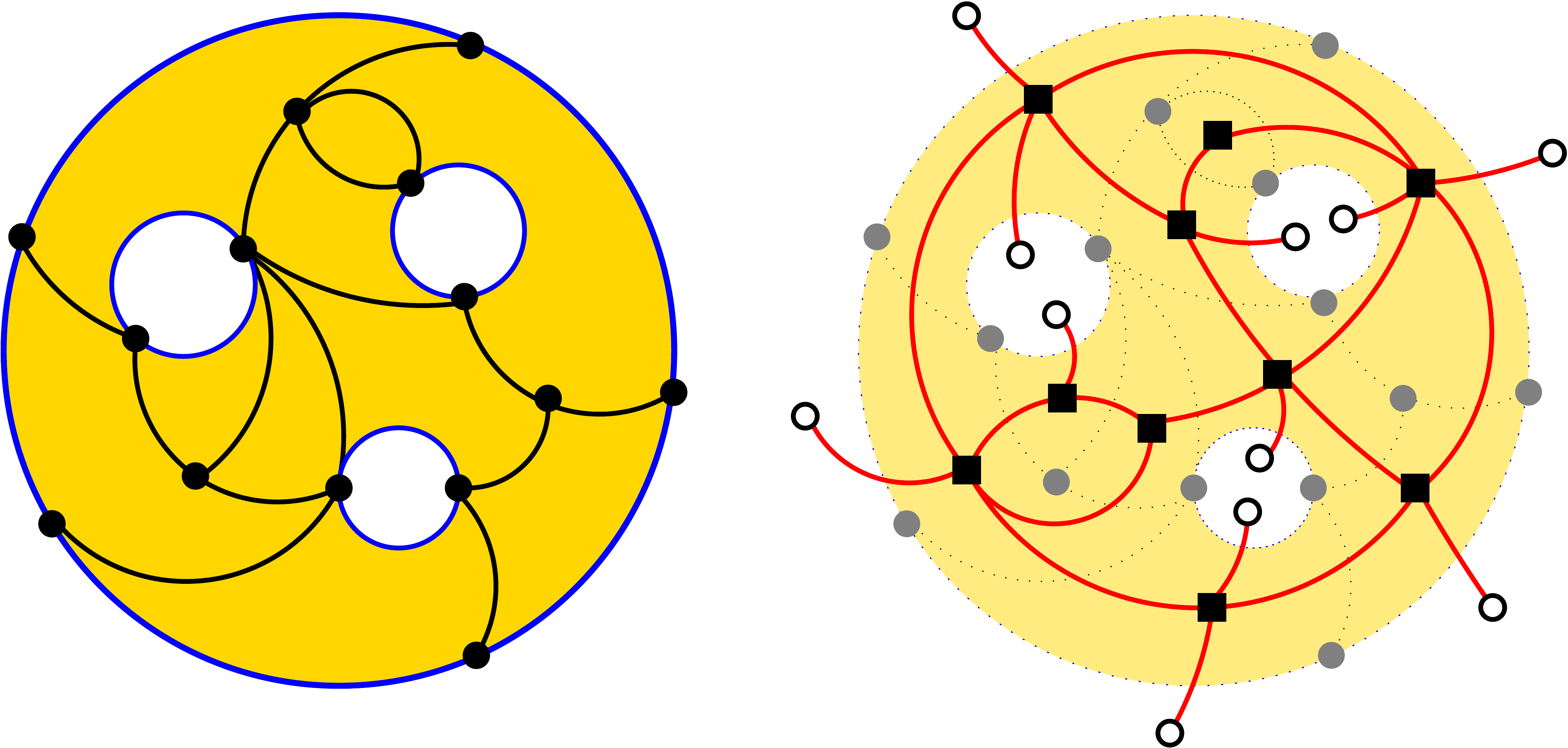}
\caption{A map with boundary and its dual.}\label{fig:dual-map}
\end{center}
\end{figure}

%
\subsection{The symbolic method and analytic combinatorics}\label{sec:prelim:symbolic}
Our main reference in enumerative combinatorics is the book of Flajolet
and Sedgewick~\cite{FlajoletSedgewig:analytic-combinatorics}. The
framework introduced in this book gives a language to translate
combinatorial conditions between combinatorial classes into equations
relating the associated generating functions. This is what is called the
\emph{symbolic method} in combinatorics. Later, we can treat these
equations as relations between analytic functions. This point of view
gives the possibility to use complex analysis techniques to obtain
information about the combinatorial classes. This is the origin of the
term \emph{analytic combinatorics}.
\\
\paragraph{\textbf{The symbolic method}.} For a set $\mA$ of objects, let $|\cdot|$ be an application (called \emph{size})
from $\mA$ to $\mathbb{N}$. We assume that the number of elements in $\mA$
with a fixed size is always finite. A pair $(\mA,|\cdot|)$ is called a
\emph{combinatorial class}. Under these assumptions, we define the formal
power series (called the \emph{generating function} or \emph{GF}
associated with the class) $\gA(x)=\sum_{a\in
\mA}x^{|a|}=\sum_{n=0}^\infty a_n x^n$. Conversely, we write
$a_{n}=[x^n]\gA(x)$. The \emph{symbolic method} provides a direct way to
translate combinatorial constructions between combinatorial classes into
equations between GFs. The constructions we use in this work and their
translation into the language of GFs are shown in Table~\ref{tab:1}.

\begin{table}[htb]
\begin{center}
\begin{tabular}{c c|c}
  \textbf{Construction} &  & \textbf{GF} \\\hline
  Union & $\mA\cup\mB$ & $\gA(x)+\gB(x)$ \\
  Product & $\mA\times\mB$ & $\gA(x)\gB(x)$ \\
  Sequence & $\Seq{\mA}$ & $\frac{1}{1-\gA(x)}$ \\
  Pointing & $\mA^\bullet$ & $x\frac{\partial}{\partial x}\gA(x)$\\
\end{tabular}\bigskip
\caption{Constructions and translations into GFs.\label{tab:1}}
\end{center}
\end{table}

The union $\mA \cup \mB$ of $\mA$ and $\mB$ refers to the disjoint union
of the classes. The cartesian product $\bA \times \mB$ of $\mA$ and $\mB$
is the set $\{(a,b): a\in \mA,\, b\in \mB\}$. The sequence $\Seq{\mA}$ of
a set $\mA$ corresponds to the set $\mE\cup\mA \cup
\left(\mA\times\mA\right) \cup \left(\mA\times\mA\times \mA\right) \cup
\ldots$, where $\mE$ denotes the empty set. At last, the pointing operator
$\mA^{\bullet}$ of a set $\mA$ consists in pointing one of the atoms of
each element $a\in \mA$. Notice that in the sequence construction, the
expression $\mE\cup\mA \cup \left(\mA\times\mA\right) \cup
\left(\mA\times\mA\times \mA\right) \cup \dots$ translates into
$\sum_{k=0}^{\infty}\gA(x)^k$, which is a sum of a geometric series. In
the case of pointing, note also that $x\frac{\partial}{\partial
z}\gA(x)=\sum_{n>0}na_n x^{n}$. 
\\
\paragraph{\textbf{Singularity analysis}.}The study of the asymptotic growth of the coefficients of GFs can be obtained by considering GFs as complex functions
 analytic around $z=0$. This is the main idea of analytic combinatorics. The growth behavior of the coefficients depends only on the smallest positive singularity of the GF. Its \emph{location} provides the \emph{exponential growth} of the coefficients, and its \emph{behavior} gives the \emph{subexponential growth} of the coefficients.

More concretely, for real numbers $R>\rho>0$ and $0<\phi<\pi/2$, let
$\Delta_\rho(\phi, R)$ be the set $\{z \in \mathbb{C} : |z|<R,\,
z\neq\rho,\, |\mathrm{Arg}(z-\rho)|>\phi\}$. We call a set of this type a
\emph{dented domain} or a \emph{domain dented} at $\rho$. Let $\gA(z)$ and
$\gB(z)$ be GFs whose smallest singularity is the real number $\rho$. We
write $\gA(z)\sim _{z\rightarrow \rho}\gB(z)$ if $\lim_{z\rightarrow
\rho}\gA(z)/\gB(z)=1$. We obtain the asymptotic expansion of $[z^n]\gA(z)$
by \emph{transfering} the behavior of $\gA(z)$ around its singularity from
a simpler function $\gB(z)$, from which we know the asymptotic behavior of
their coefficients. This is the main idea of the so-called \emph{Transfer
Theorems} developed by Flajolet and Odlyzko~\cite{FlaOdl}. These results
allows us to deduce asymptotic estimates of an analytic function using its
asymptotic expansion near its dominant singularity. In our work we use a
mixture of Theorems VI.1 and VI.3
from~\cite{FlajoletSedgewig:analytic-combinatorics}:
\begin{prop}[Transfer Theorem]\label{thm:transfer}
If $\gA(z)$ is analytic in a dented domain $\Delta=\Delta_\rho(\phi,R)$,
where $\rho$ is the smallest singularity of $\gA(z)$, and
$$\gA(z) \underset{z\in\Delta, z \rightarrow \rho}{\sim} c\cdot \left(1-\frac{z}{\rho}\right)^{-\alpha}+O\left(\left(1-\frac{z}{\rho}\right)^{-\alpha+\gamma}\right),$$
for $\alpha\not\in\{0,-1, -2,\ldots\}$, and $\gamma>0$ then
\begin{equation}\label{eq:transfer}
a_n\ = c\cdot
\frac{n^{\alpha-1}}{\Gamma(\alpha)}\cdot\rho^{-n}\left(1+O(n^{-\gamma})\right),\end{equation}
where $\Gamma$ is the Gamma function:
$\Gamma(u)=\int_{0}^{\infty}t^{u-1}e^{-t}dt$.
\end{prop}

\section{Non-crossing partitions on surfaces with boundary}\label{sec:enumeration}

In this section we introduce the precise definition of a non-crossing
partition on a surface with boundary. The notion of a non-crossing
partition on a general surface is not as simple as in the case of a disk,
and must be stated in terms of objects more general than maps. Our
strategy to obtain asymptotic estimates for the number of non-crossing
partitions on surfaces consists in showing that we can restrict ourselves
to the study of certain families of maps. More concretely, we show that
the study of non-crossing partitions is a particular case of the study of
hypermaps~\cite{cori}, which can be interpreted as bipartite maps. The
plan for this section is the following: in
Subsection~\ref{subsection:definitions-non-crossing-partition} we set up
our notation and we define a non-crossing partition on a general surface.
In Subsection~\ref{subsect:reduction-bipartite-maps} we show that we can
restrict ourselves to the study of bipartite maps in which vertices belong
to the boundary of the surface.

\subsection{Bipartite subdivisions and non-crossing partitions}\label{subsection:definitions-non-crossing-partition}

Let $\Sigma$ be a connected surface with boundary, and let
$\mathbb{S}^1_1,\mathbb{S}^1_2,\dots, \mathbb{S}^1_{\beta(\Sigma)}$ be the
connected components of the boundary of $\Sigma$.

A \emph{bipartite subdivision} $S$ of $\Sigma$ with $n$ vertices is a
decomposition of $\Sigma$ into zero-, one-, and two-dimensional open and
connected subsets, where the $n$ vertices lay on the boundary of $\Sigma$,
and there is a two-coloring (namely, using black and white colors) of the
two-dimensional regions, such that each vertex is incident (possibly more
than once) with a unique black two-dimensional region.  We use the
notation $A(S)$ to denote the set $A_{1}\cup \cdots \cup
A_{\beta(\Sigma)}$ of vertices of $S$.

\begin{figure}[htb]
\begin{center}
\scalebox{0.356 }{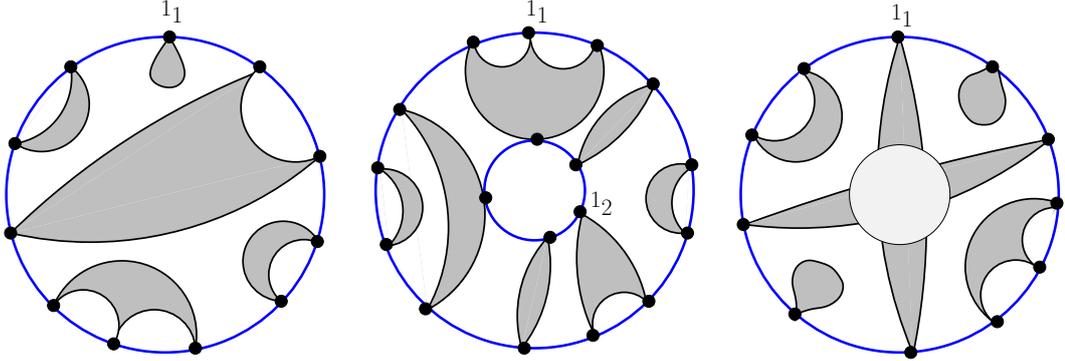}
\caption{Geometric representation of non-crossing partitions on a disk, on
a cylinder, and on a M\"obius band.}\label{fig:intro-surfaces}
\end{center}
\end{figure}

For $1\leq r\leq \beta(\Sigma)$, let $A_r=\{1_{\{r\}},2_{\{r\}},\dots,
n_{\{r\}}\}$ be the set of vertices on $\mathbb{S}^1_r$, i.e.,
$A_{r}=A(S)\cap \mathbb{S}^{1}_{r}$. Vertices on each boundary are labeled
in counterclockwise order, and satisfy the property that
$|A_1|+|A_2|+\dots+|A_{\beta(\Sigma)}|=n$. In particular, boundary
components are distinguishable.
Observe that an equivalent way to label these vertices is distinguishing
on each boundary component an edge-root, whose ends are vertices
$1_{\{r\}}$ and $2_{\{r\}}$.

In general, bipartite subdivisions are not maps: two-dimensional subsets
could not be homeomorphic to open disks. Black faces on a bipartite
subdivision are called \emph{blocks}. A block of size $k$ is
\emph{regular} if it is incident with exactly $k$ vertices and it is
contractible (i.e., homeomorphically equivalent to a disk).
A bipartite subdivision is \emph{regular} if each block is regular. All
bipartite subdivisions are \emph{rooted}: every connected component of the
boundary of $\Sigma$ is edge-rooted in counterclockwise order. We denote
by $\mmS_{\Sigma}(n)$ and $\mR_{\Sigma}(n)$ the set of general and regular
bipartite subdivisions of $\Sigma$ with $n$ vertices, respectively.
Observe that the total number of vertices is distributed among all the
components of the boundary of $\Sigma$. In particular, it is possible that
a boundary component is not incident with any vertex. See
Figure~\ref{fig:non-crossing} for examples of bipartite subdivisions. In
particular, the darker blocks in the first bipartite subdivision are not
regular.

\begin{figure}[h!]
\begin{center}
\scalebox{0.375 }{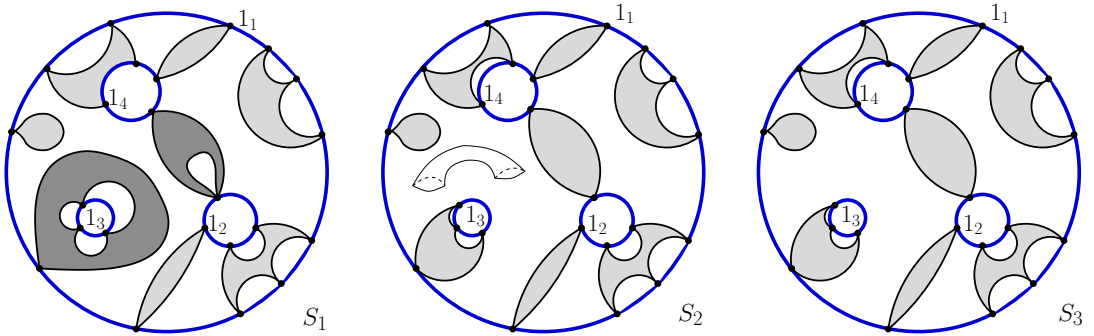}
\end{center}
\caption{Three bipartite subdivisions $S_{1},S_{2}$ and $S_{3}$. $S_{2}$
is regular but not reducible while $S_{3}$ is irreducible.}
\label{fig:non-crossing}
\end{figure}

Let $S$ be a bipartite subdivision $S$ of $\Sigma$ with $n$ vertices and
let $X_{1},\ldots,X_{r}$ the set ${\mathcal{S}}$ be its blocks. Clearly,
these blocks define the partition $\pi_{\Sigma}\left(S\right)=\{X_{1}\cap
A,\ldots,X_{r}\cap A\}$ of the vertex set  $A=A_1 \cup \dots \cup
A_{\beta(\Sigma)}$. We say that a partition of $A$ is {\em non-crossing}
if  it is equal to $\pi_{\Sigma}\left(S\right)$  for some  bipartite
subdivision $S$ of $\Sigma$.   A non-crossing partition is said to be
\emph{regular} if it arises from a regular bipartite subdivision. Observe
that this definition generalizes the notion of a non-crossing partition on
a disk. We define $\Pi_\Sigma(n)$ as the set of non-crossing partitions of
$\Sigma$ with $n$ vertices and we set $C_{\Sigma}=|\Pi_\Sigma(n)|$. Notice
that this definition of $\Pi_{\Sigma}$ is equivalent to the one we gave in
the introduction. In the rest of the paper we  adopt the new definition as
this will simplify the presentation of our results and proofs.
%

\subsection{Reduction to the map framework}\label{subsect:reduction-bipartite-maps}
In this subsection we show that we can restrict ourselves to the study of
bipartite maps in which vertices belong to the boundary of the surface.
Later, this reduction will allow to study non-crossing partitions in the
context of map enumeration.

Let $\Sigma_1$ and $\Sigma_2$ be surfaces with boundary. We write
$\Sigma_2 \subset \Sigma_1$ if there exists a continuous injection
$i:\Sigma_2\hookrightarrow \Sigma_1$ such that $i(\Sigma_2)$ is
homeomorphic to $\Sigma_2$. If $S$ is a bipartite subdivision of
$\Sigma_2$ and $\Sigma_2 \subset \Sigma_1$, then the injection $i$ induces
a bipartite subdivision $i(S)$ on $\Sigma_1$ such that
$\pi_{\Sigma_2}(S)=\pi_{\Sigma_1}\left(i(S)\right)$. Roughly speaking, all
bipartite subdivisions on $\Sigma_2$ can be realized on a surface
$\Sigma_1$ which contains $\Sigma_2$. One can write then that
$\Pi_{\Sigma_2}(n)\subseteq \Pi_{\Sigma_1}(n)$ if $\Sigma_2\subset
\Sigma_1$, and then it holds that
$|\Pi_{\Sigma_2}(n)|\leq|\Pi_{\Sigma_1}(n)|$. This proves the trivial
bound $C(n) \leq |\Pi_{\Sigma}(n)|$ for all choices of $\Sigma$.

As the following lemma shows, regularity is conserved by injections of
surfaces.
\begin{lemma}\label{lemma:regularitat}
Let $M_1$ be a regular bipartite subdivision of $\Sigma_1$, and let
$\Sigma_1 \subset \Sigma$. Then $M_1$ defines a regular bipartite
subdivision $M$ over $\Sigma$ such that
$\pi_{\Sigma_1}\left(M_1\right)=\pi_{\Sigma}(M)$.
\end{lemma}
\begin{proof} Let $i:\Sigma_1 \hookrightarrow \Sigma$ be the corresponding injective application, and consider $M=i\left(M_1\right)$. In particular, a block $\pi$ of $M_1$ is topologically equivalent to the block $i(\pi)$: $i$ is a homeomorphism between $\Sigma$ and $i(\Sigma)$. Hence $i(\pi)$ is an open contractible set and $M$ is regular.
\end{proof}
The following proposition allows us to reduce the problem to the study of
regular bipartite subdivisions.
\begin{lemma} \label{lemma:only-regular}
Let $S\in \mmS_{\Sigma}(n)$ be a bipartite subdivision of $\Sigma$ and let
$\pi_\Sigma(S)$ be the associated non-crossing partition on $\Sigma$.
Then, there exists a regular bipartite subdivision $R\in\mR_{\Sigma}(n)$
such that $\pi_{\Sigma}(R)=\pi_{\Sigma}(S)$.
\end{lemma}
\begin{proof} Each bipartite subdivision has a finite number of blocks. For each block we will apply a finite
number of transforms in order to change it into a regular block, without
changing the associated non-crossing partition. We consider two cases
according to whether the block studied is contractible or not.

Let $f$ be a contractible block of $S$. Suppose that the boundary of $f$
consists of more than one connected component. We define the operation of
\emph{joining boundaries} as follows: let $l$ be a path that joins a
vertex $v$ in one component of the boundary of $f$ with a vertex $u$ in
another component of the boundary. This path exists because $f$ is a
connected and open subset of $\Sigma$. Consider also two paths $l_1, l_2$
that join these two vertices around the initial path $l$, as illustrated
in Figure~\ref{fig:join-cut}. Note that these paths $l_1$ and $l_2$ also
exist since we are dealing with open subsets. We define the new block
$f_1$ as the one obtained from the initial block $f$ by deleting the face
defined by $l_1$ and $l_2$ which contains $l$ (see the leftmost part of
Figure~\ref{fig:join-cut} for an example). Let $s_1$ be the resulting
bipartite subdivision. Observe that the number of connected components of
the boundary of $f_1$ is the same as for $f$ minus one. We can apply this
argument over $f$ as many times as the number of components of the
boundary of $f$ is strictly greater than one. At the end, we obtain a
bipartite subdivision with the same induced non-crossing partition, such
that the block derived from $f$ has exactly one boundary component.

Suppose now that the boundary of the block $f$ has a single component, but
it is not simple. Let $v$ be a vertex incident $p>1$ times with $f$. In
this case we define the operation of \emph{cutting a vertex} as follows:
consider the intersection of a small ball of radius $\epsilon > 0$
centered at $v$ with the block $f$, namely $B_{\epsilon}(v) \cap f$.
Observe that $\left(B_{\epsilon}(v) \cap f\right)\backslash \{v\}$ has
exactly $p$ connected components. We define the new block by deforming
$p-1$ of these components in such a way that they do not intersect the
boundary of $\Sigma$. Next, we paste the vertex $v$ to the unique
component which has not been deformed (see the rightmost part of
Figure~\ref{fig:join-cut} for an example). Then the resulting bipartite
subdivision has the same associated non-crossing partition, and $v$ is
incident with the corresponding block exactly once. Applying this argument
for each vertex of $f$ we get a block with a single simple boundary.

\begin{figure}[h!]
\begin{center}
\includegraphics[scale=.057]{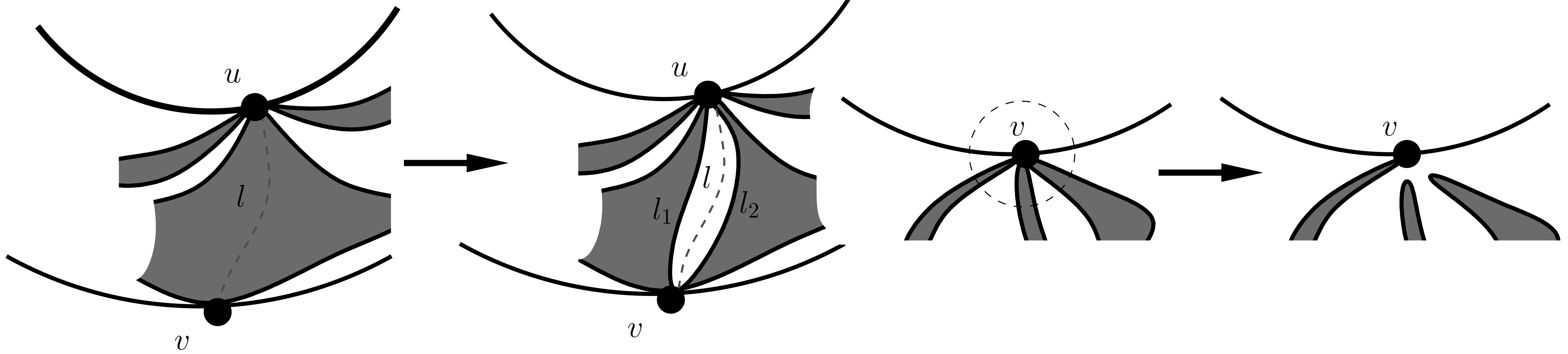}
\end{center}
\bigskip \caption{The operations of joining boundaries and cutting vertices.}
\label{fig:join-cut}
\end{figure}

Summarizing, from each contractible block $f$ of $S$ we construct a new
block $f'$ which is incident with the same vertices as $f$.

To conclude, suppose now that $f$ is an non-contractible block of $S$. Let
$\mathbb{S}^1_f$ be a non-contractible cycle contained in $f$. We cut the
surface along this cycle. We paste either a disk or a pair of disks along
the border depending on whether $\mathbb{S}^1_f$ is one- or two-sided.
This operation either increases the number of connected components or
decreases the genus of the surface.

Observe that the number of times we need to apply this operation is
bounded by $-\chi(\Sigma)$; in particular, it is finite. At the end, after
converting each block to a contractible one, all blocks are contractible
and the resulting surface (possibly with many connected components) is
$\Sigma_1 \subset \Sigma$. The resulting bipartite subdivision $M'$ on
$\Sigma_1$ is regular (since all the blocks are regular), and then by
Lemma~\ref{lemma:regularitat} there exists a regular bipartite subdivision
$R$ over $\Sigma$ such that
$\pi_{\Sigma}\left(R\right)=\pi_{\Sigma_1}\left(M'\right)$, as claimed.
\end{proof}

Notice that we just proved the following.
\begin{eqnarray}
|\Pi_{\Sigma}(n)|\leq |\mathcal{R}_{\Sigma}(n)|.\label{eooj}
\end{eqnarray}

We say that a bipartite subdivision $M$ is \emph{irreducible} in
$\Sigma_1$ if the associated non-crossing partition $\pi_{\Sigma_1}(M)$ is
regular and all its white faces are contractible.
%
%
%
In this case, we also say that the non-crossing partition
$\pi_{\Sigma_1}(M)$ is irreducible. We denote by $\mathcal{P}_{\Sigma}(n)$
the set of irreducible bipartite subdivisions. Clearly, the following
holds.
\begin{eqnarray}
|\mathcal{P}_{\Sigma}(n)| \leq |\mathcal{R}_{\Sigma}(n)|. 
\label{mltt}
\end{eqnarray}

The following lemma is a basic consequence of the previous discussions,
and allows us to reduce our study to the enumeration in the context of
maps.

\begin{lemma}\label{lemma:irreducibility}
Let $M$ be an irreducible bipartite subdivision of $\Sigma$. Then the
two-dimensional regions of $M$ are all contractible (hence, faces).
\end{lemma}

\begin{proof} From Lemma~\ref{lemma:only-regular},
we only need to deal with white two-dimensional regions. For a white face
whose interior is not homeomorphic to an open disk, there exists a
non-contractible cycle $\mathbb{S}^1$. Cutting along $\mathbb{S}^1$ we
obtain a surface $\Sigma'$ such that $\Sigma' \subset \Sigma$ and $M$ is
induced in $\Sigma'$, a contradiction. As a conclusion, all faces are
contractible.
\end{proof}

The above lemma says that irreducible bipartite subdivisions define
bipartite maps. In the next section we reduce our study to the family of
irreducible bipartite subdivisions. This permits us to upper-bound
$|\mathcal{P}_{\Sigma}(n)|$ instead of dealing with the more complicated
task of upper-bounding $|\mathcal{R}_{\Sigma}(n)|$. The reason why this
also gives an asymptotic bound for $|\Pi_{\Sigma}(n)|$ is that the
subfamily $\mathcal{P}_{\Sigma}(n)$ provides the main contribution to the
asymptotic estimates for $\mathcal{R}_{\Sigma}(n)$. Therefore~\eqref{mltt}
can be seen, asymptotically, as an equality, and, that way, the result
follows from~\eqref{eooj}.

%


\section{Upper bounds for non-crossing partitions on surfaces}\label{sec:upper-bounds}

The plan for this section is the following: in
Subsection~\ref{apen:3-enum} we introduce families of plane trees that
arise by duality on non-crossing partitions on a disk. These combinatorial
structures are used in Subsection~\ref{section:enumeration} to obtain a
tree-like structure which provides a way to obtain asymptotic estimates
for the number of irreducible bipartite subdivisions of $\Sigma$ with $n$
vertices, $|\mathcal{P}_{\Sigma}(n)|$. These asymptotic estimates are
found in Subsection~\ref{apen:3.5-enum} for irreducible bipartite
subdivisions.  Finally, we prove in Subsection~\ref{apen:4-enum} that the
number of irreducible bipartite subdivisions is asymptotically equal to
the number of bipartite subdivisions, hence the estimate obtained in
Subsection~\ref{section:enumeration} is an upper bound for the number of
non-crossing partitions on surfaces. All previous steps are summarized in
Subsection~\ref{subsection:final-result}.

\subsection{Planar constructions} \label{apen:3-enum}
The dual map of a non-crossing partition on a disk is a tree, which is
called the (non-crossing partition) tree associated with the non-crossing
partition. This tree corresponds to the notion of dual map for surfaces
with boundary introduced in Subsection~\ref{sec:prelim:embedded}. Recall
that vertices of degree one are called the \emph{dangling leaves} of the
tree. Vertices of the tree are called \emph{block} vertices if they are
associated with a block of the non-crossing partition. The remaining
vertices are either \emph{non-block} vertices or danglings.  By
construction, all vertices adjacent to a block vertex are non-block
vertices. Conversely, each vertex adjacent to a non-block vertex is either
a block vertex or a dangling. Graphically, we use the symbols
$\blacksquare$ for block vertices, $\square$ for non-block vertices and
$\circ$ for danglings. Non-crossing partitions trees are rooted: the root
of a non-crossing partition tree is defined by the root of the initial
non-crossing partition on a disk. The block vertex which carries the role
of the root vertex of the tree is the one associated with the block
containing vertex with label $2$ (or equivalently, the end-vertex of the
root). See Figure~\ref{fig:partition-tree} for an example of this
construction.

\begin{figure}[h!]
\begin{center}
\includegraphics[scale=1.5]{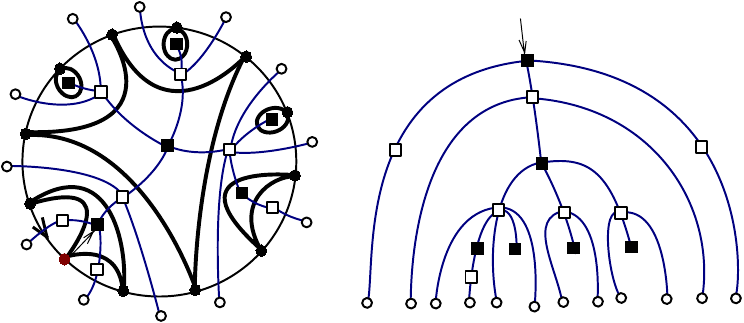}
\caption{A non-crossing partition on a disk and the associated
non-crossing partition tree.\label{fig:partition-tree}}
\end{center}
\end{figure}

Let $\mathcal{T}$ be the set of non-crossing partitions trees, and let
$\nT=\nT(z,u)=\sum_{n,m\geq 0} t_{n,m} z^n u^m$ be the corresponding
generating function. The variable $z$ marks danglings and $u$ marks block
vertices. We use an auxiliary family $\mathcal{B}$, defined as the set of
trees which are rooted at a non-block vertex. Let
$\nB=\nB(z,u)=\sum_{n,m\geq 0}b_{n,m} z^n u^m$ be the associated
generating function. The next lemma gives the exact enumeration of
$\mathcal{T}$ and $\mathcal{B}$. In particular, this lemma implies the
well-known Catalan numbers for non-crossing partitions on a disk.

\begin{lemma}\label{lemma:planar-trees}
The number of non-crossing trees counted by the number of danglings and
block vertices is enumerated by the generating function
\begin{equation}\label{eq:gen-catalan-function}
\nT(z,u)=\frac{1-z(1-u)-\sqrt{(z(1-u)-1)^2-4zu}}{2zu}.
\end{equation}
Furthermore, $\nB(z,u)=z\nT(z,u)$.
\end{lemma}
\begin{proof} We establish combinatorial relations between $\mathcal{B}$ and $\mathcal{T}$ from which we
deduce the result. Observe that there is no restriction on the number of
vertices incident with a given block. Hence the degree of every block
vertex is arbitrary. This condition is translated symbolically via the
relation
$$\mathcal{T}=\{\blacksquare\} \times \mathrm{Seq}\left(\mathcal{B}\right).$$
Similarly, $\mathcal{B}$ can be written in the form
$$\mathcal{B}=\{\circ\}\times\mathrm{Seq}\left(\mathcal{T}\times \{\circ\} \right).$$
These combinatorial conditions translate using Table~\ref{tab:1} into the
system of equations
$$\nT=\frac{u}{1-\nB},\ ,\ \ \nB= \frac{z}{1-z\nT}.$$
Substituting the expression of $\nB$ in the first equation, one obtains
that $\nT$ satisfies the relation $z\nT^2+(z(1-u)-1)\nT+u=0$. The solution
to this equation with positive coefficients
is~(\ref{eq:gen-catalan-function}). Solving the previous system of
equations in terms of $\nB$ brings $\nB=z \nT$, as claimed. \end{proof}

Observe that writing $u=1$ in $\nT$ and $\nB$ we obtain that
$\nT(z)=\nT(z,1)=\frac{1-\sqrt{1-4z}}{2z}$, and $\nB(z)=\nB(z,1)=z\nT(z)$,
deducing the well-known generating function for Catalan numbers.

We introduce another family of trees related to non-crossing partitions
trees, which we call \emph{double trees}. A double tree is defined in the
following way: consider a path where we concatenate block vertices and
non-block vertices. We consider the internal vertices of the path. A
double tree is obtained by pasting on every block vertex of the path a
pair of elements of $\mathcal{T}$ (one at each side of the path), and a
pair of elements of $\mathcal{B}$ for non-block vertices. We say that a
double tree is of type either $\blacksquare-\blacksquare$,
$\blacksquare-\square$, or $\square-\square$ depending on the ends of the
path. An example for a double tree of type $\blacksquare-\blacksquare$ is
shown in Figure~\ref{fig:double-tree}.

\begin{figure}[htb]
\begin{center}
\includegraphics[scale=0.33]{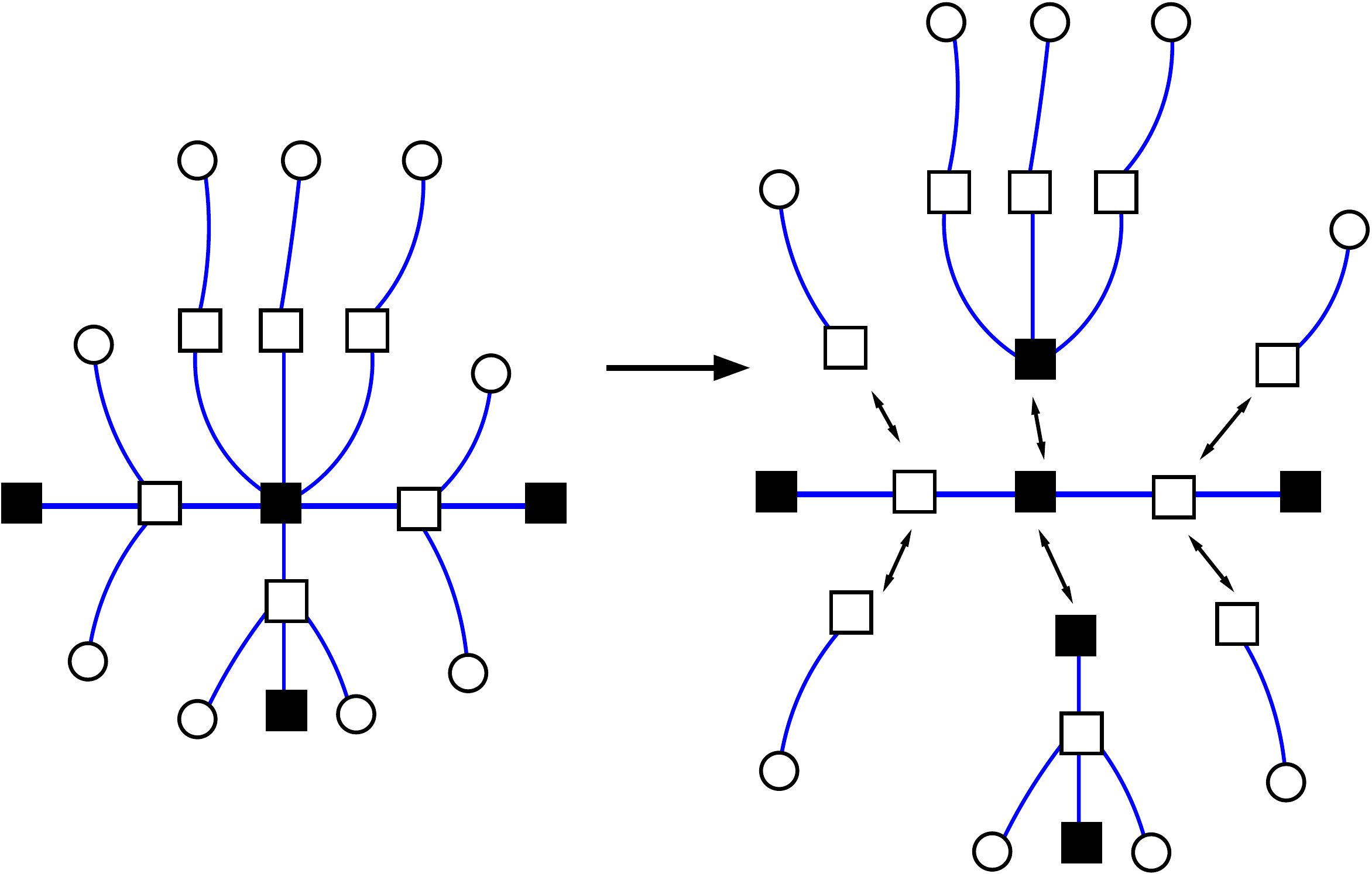}
\caption{A double tree and its decomposition. \label{fig:double-tree}}
\end{center}
\end{figure}

We denote these families by $\mathcal{T}_{\blacksquare-\blacksquare}$,
$\mathcal{T}_{\square-\blacksquare}$, and $\mathcal{T}_{\square-\square}$,
and the corresponding generating function by $\nT_{1}(z,u)=\nT_{1}$,
$\nT_{2}(z,u)=\nT_{2}$, and $\nT_{3}(z,u)=\nT_{3}$, respectively. Recall
that in all cases $z$ marks danglings and $u$ marks block vertices. A
direct application of the symbolic method provides a way to obtain
explicit expressions for the previously defined generating functions. The
decomposition and the GFs of the three families are summarized in
Table~\ref{table:double-trees}.

{\renewcommand{\arraystretch}{1.2}
\begin{table}[htb]
\begin{center}
\begin{tabular}{c|c|c|c}
  \textbf{Family} & \textbf{Specification}&\textbf{Development} & \textbf{Compact expression} \\\hline
$\mathcal{T}_{\square-\blacksquare}$ & $\mathrm{Seq}\left(\mathcal{B}^2 \times \mathcal{T}^2 \right)$            &$1+\frac{1}{u}\nB^2\nT^2+\frac{1}{u^2}\nB^4\nT^4+\dots$& $1/(1-\nT^2\nB^2/u)$ \\
$\mathcal{T}_{\blacksquare-\blacksquare}$ &  $\mathcal{B}^2\times \mathrm{Seq}\left(\mathcal{B}^2 \times \mathcal{T}^2 \right)$         & $\nB^2+\frac{1}{u}\nB^4\nT^2+\frac{1}{u^2}\nB^6\nT^4+\dots$ & $\nB^2/(1-\nT^2\nB^2/u)$ \\
$\mathcal{T}_{\square-\square}$ &  $\mathcal{T}^2\times
\mathrm{Seq}\left(\mathcal{B}^2 \times \mathcal{T}^2 \right)$
&$\frac{1}{u}\nT^2+\frac{1}{u^2}\nB^2\nT^4+\frac{1}{u^3}\nB^4\nT^6+\dots$
&
$\frac{1}{u}\nT^2/(1-\nT^2\nB^2/u)$\\
\end{tabular}\bigskip
\caption{GFs for double trees.\label{table:double-trees}}
\end{center}
\end{table}}

To conclude, the family of \emph{pointed} non-crossing trees
$\mathcal{T}^{\bullet}$ is built by pointing a dangling on each
non-crossing partition tree. In this case, the associated GF is
$\nT^{\bullet}=z \frac{\partial}{\partial z}\nT$. Similar definitions can
be done for the family $\mathcal{B}$. Pointing a dangling defines a unique
path between this distinguished dangling and the root of the tree.

\subsection{The scheme of an irreducible bipartite subdivision.}\label{section:enumeration}
In this subsection we generalize the construction of non-crossing
partition trees introduced in Subsection~\ref{apen:3-enum}. In order to
characterize it, we exploit the dual construction for maps on surfaces
(see Subsection~\ref{subs:maps-duality}). More concretely, for an element
$M\in\mathcal{P}_{\Sigma}(n)$, let $M^{*}$ be the dual map of $M$ on
$\overline{\Sigma}$. By construction, there is no incidence in $M^{*}$
between either pairs of block vertices or pairs of non-block vertices.

From $M^{*}$ we define a new rooted map (a root for each boundary
component of $\Sigma$) on $\overline{\Sigma}$ in the following way: we
start by deleting recursively vertices of degree one which are not roots.
Then we continue \emph{dissolving} vertices of degree two, that is,
replacing the two edges incident to a vertex of degree two with a single
edge. The resulting map has $\beta(\Sigma)$ faces and all vertices have
degree at least three (apart from root vertices, which have degree one),
and vertices of two colors (vertices of different colors could be
end-vertices of the same edge). The resulting map is called the
\emph{scheme associated} with $M$; we denote it by $\mathfrak{s}_M$.  See
Figure~\ref{fig:fig2} for an example of this construction.

\begin{figure}[h!]
\begin{center}
\includegraphics[scale=0.18]{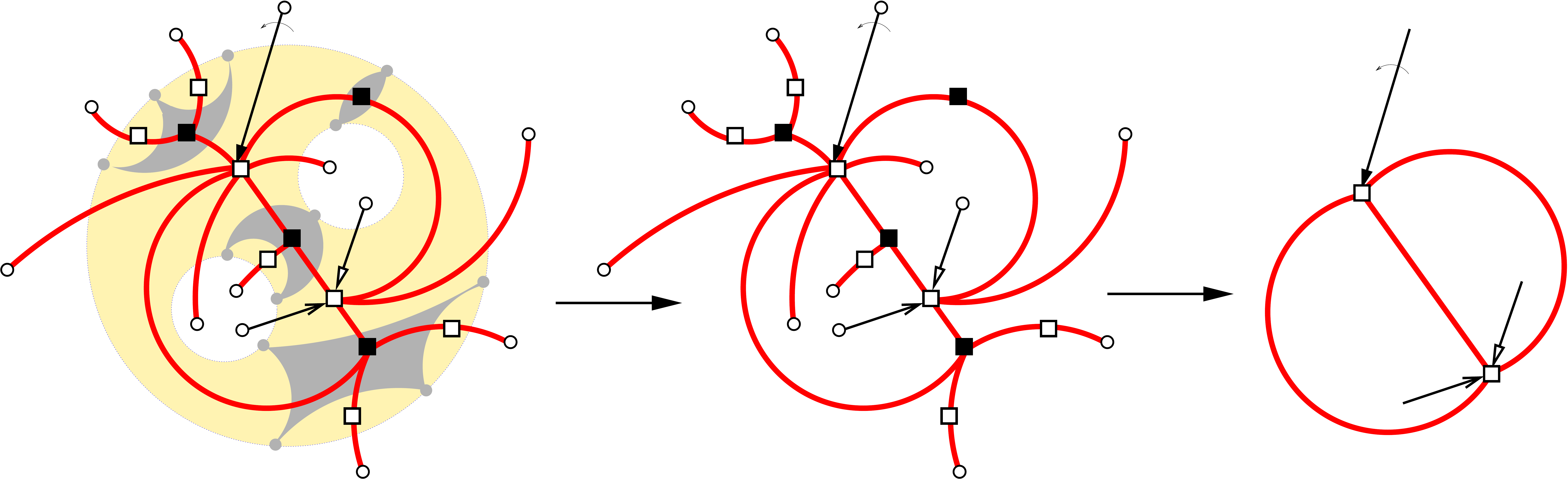}
\end{center}
\caption{The construction of the scheme of an element in
$\mathcal{P}_{\Sigma}$. We consider the dual of an irreducible bipartite
subdivision (leftmost figure). After deleting vertices of degree one
recursively and dissolving vertices of degree two, we obtain the
associated scheme (rightmost figure).\label{fig:fig2}}
\end{figure}

The previous decomposition can be constructed in the reverse way: duals of
irreducible bipartite subdivision are constructed from a generic scheme
$\mathfrak{s}$ in the following way.

\begin{enumerate}
\item For an edge of $\mathfrak{s}$ with both end-vertices of type $\blacksquare$, we paste a double tree of type $\blacksquare-\blacksquare$ along it. Similar operations are done for edges with end-vertices $\{\square,\blacksquare\}$ and $\{\square,\square\}$.
\item For vertex $v$ of $\mathfrak{s}$ of type $\blacksquare$ we paste $\dd(v)$ elements of $\mathcal{T}$ (identifying the roots of the trees with $v$), one on each corner of $v$. The same operation is done for vertices of $\mathfrak{s}$ of type $\square$.
\item We paste an element of $\mathcal{T}^{\bullet}$ along each one of the roots of $\mathfrak{s}$ (the marked leaf determines the dangling root).
\end{enumerate}

To conclude, this construction provides a way to characterize the set of
schemes. Indeed, if we denote by $\mathfrak{S}_\Sigma$ the set of maps on
$\overline{\Sigma}$ with $\beta(\Sigma)$ faces with a root on each face
and with vertices of two different colors (namely, vertices of type
$\blacksquare$ and $\square$), then $|\mathfrak{S}_\Sigma|$ is finite,
since the number of faces of each element in $\mathfrak{S}_\Sigma$ is
equal to $\beta(\Sigma)$. In fact, $\mathfrak{S}_\Sigma$ is the set of all
possible schemes: from an arbitrary element $\mathfrak{s}\in
\mathfrak{S}_\Sigma$ we can construct a map on $\overline{\Sigma}$ with
$\beta\left(\Sigma\right)$ faces by pasting double trees along each edge
of $\mathfrak{s}$ (according to the end-vertices of each edge). In other
words, given $M\in\mathcal{P}_{\Sigma}(n)$ and $\mathfrak{s}_M$, $M^{*}$
can be reconstructed by pasting on every edge of $\mathfrak{s}_M$ a double
tree, depending on the nature of the end-vertices of each edge of
$\mathfrak{s}_M$. See Figure~\ref{fig:dual2} for an example.

\begin{figure}[htb]
\begin{center}
\includegraphics[scale=0.19]{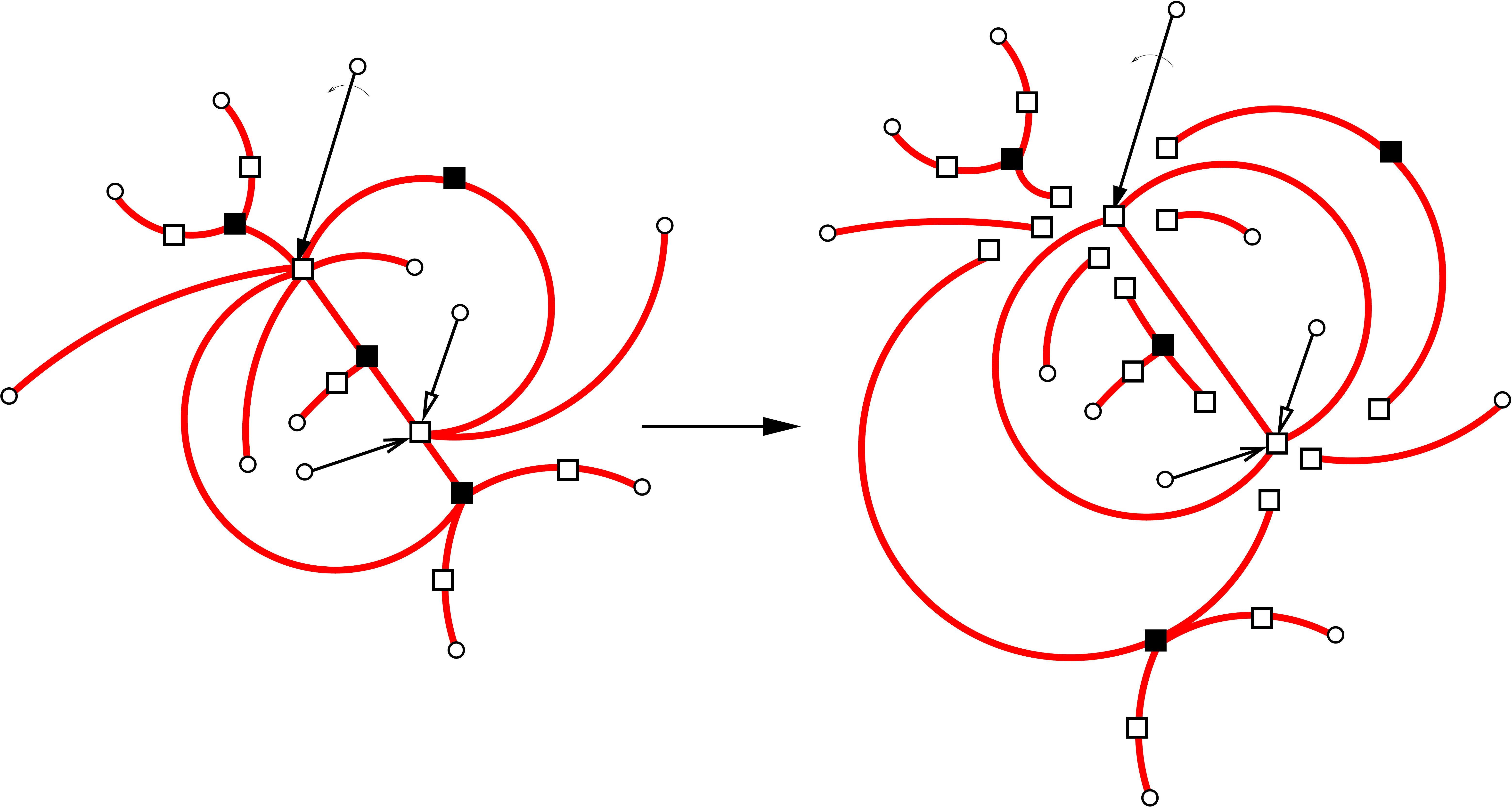}
\caption{The decomposition into bicolored trees and the associated
scheme.\label{fig:dual2} }
\end{center}
\end{figure}

\subsection{Asymptotic enumeration} \label{apen:3.5-enum}
The decomposition introduced in Subsection~\ref{section:enumeration} can
be exploited in order to get asymptotic estimates for
$|\mathcal{P}_{\Sigma}(n)|$, and consequently upper bounds for
$|\Pi_{\Sigma}(n)|$. In this subsection we provide estimates for the
number of irreducible bipartite subdivisions. We obtain these estimates
directly for the surface $\Sigma$, while the usual technique consists in
reducing the enumeration to surfaces of smaller genus, and returning back
to the initial one by topological ``pasting'' arguments. The main point
consists in exploiting tree structures of the dual graph associated with
an irreducible bipartite subdivision. The main ideas are inspired
by~\cite{BernardiRue:triangulations-higher-genus}, where the authors find
the asymptotic enumeration of simplicial decompositions of surfaces with
boundaries without interior points.

We use the notation and definitions introduced in
Subsection~\ref{apen:3-enum} (i.e., families of trees, double trees and
pointed trees, and the corresponding GFs), joint with the decomposition
introduced in Subsection~\ref{section:enumeration}. Let us now introduce
some extra notation.

We denote by $\mathcal{P}_\Sigma(n,m)$ the set of irreducible bipartite
subdivisions of $\Sigma$ with $n$ vertices and $m$ blocks. We write
$p_{n,m}^{\Sigma}$ for the cardinality of this set and
$\textbf{P}_{\Sigma}(z,u)=\sum_{n,m\geq 0}p_{n,m}^{\Sigma} z^n u^m$. Let
$p_{n}^{\Sigma}=\sum_{m\geq
0}p_{n,m}^{\Sigma}=[z^n]\textbf{P}_{\Sigma}(z,1)$. Let
$\mathfrak{s}\in\mathfrak{S}_\Sigma$. Denote by $v_1(\mathfrak{s})$ and
$v_2(\mathfrak{s})$  the set of vertices of type $\blacksquare$ and
$\square$ of $\mathfrak{s}$, respectively. Write $b(\mathfrak{s}),
w(\mathfrak{s})$ for the number of roots which are incident with a vertex
of type $\blacksquare$ and $\square$, respectively. In particular,
$b(\mathfrak{s})+w(\mathfrak{s})=\beta(\Sigma)$. Denote by
$e_{1}(\mathfrak{s})$ the number of edges in $\mathfrak{s}$ of type
$\blacksquare-\blacksquare$. We similarly define $e_2(\mathfrak{s})$ and
$e_3(\mathfrak{s})$ for edges of type $\square-\blacksquare$ and
$\square-\square$, respectively. Observe that
$e_1(\mathfrak{s})+e_2(\mathfrak{s})+e_3(\mathfrak{s})+b(\mathfrak{s})+w(\mathfrak{s})$
is equal to the number of edges of $\mathfrak{s}$, $e(\mathfrak{s})$. For
a vertex $v$ of $\mathfrak{s}$, denote by $r(\mathfrak{s})$ the number of
roots which are incident with it. Finally, denote by
$\mathfrak{C}_{\Sigma} \subset \mathfrak{S}_{\Sigma}$ the set of maps on
$\mathfrak{S}_{\Sigma}$ whose vertex degree is equal to three (namely,
cubic maps on $\overline{\Sigma}$ with $\beta(\Sigma)$ faces).

The decomposition discussed in Subsection~\ref{section:enumeration}
together with Proposition~\ref{thm:transfer} give the following:

\begin{lemma}\label{thm:assympt-enumeration-1}
Let $\Sigma$ be a surface with boundary. Then
\begin{equation}\label{eq:asympt-enumeration-2}
[z^n]\nP_\Sigma(z,1)=p_{n}^{\Sigma}=\frac{c(\Sigma)}{\Gamma\left(-3\chi(\Sigma)/2+\beta(\Sigma\right))}
\cdot n^{-3\chi(\Sigma)/2+\beta(\Sigma)-1} \cdot 4^n
\left(1+O\left(n^{-1/2}\right)\right),
\end{equation}
where $c(\Sigma)$ is a function depending only on $\Sigma$.
\end{lemma}

\begin{proof} According to the decomposition introduced in Subsection~\ref{section:enumeration}, $\nP_\Sigma(z,u)$ can be written in the following form: for each $\mathfrak{s}\in \mathfrak{S}_{\Sigma}$, we replace edges (not roots) with double trees, roots with pointed trees, and vertices with sets of trees. More concretely,
\begin{eqnarray}\label{eq:sing-u}
\textbf{P}_{\Sigma}(z,u)=\sum_{\mathfrak{s}\in \mathfrak{S}_\Sigma} && u^{\left|v_1(\mathfrak{s})\right|} \nT_{1}^{e_1(\mathfrak{s})}\nT_{2}^{e_2(\mathfrak{s})}\nT_{3}^{e_3(\mathfrak{s})}\left(\frac{\nT}{u}\right)^{\sum_{x\in v_1(\mathfrak{s})}(\dd(x)-2r(x))} \times\\
&&\nB^{\sum_{y\in
v_2(\mathfrak{s})}(\dd(y)-2r(y))}\left(\frac{\nT^{\bullet}}{u}\right)^{B(\mathfrak{s})}\left(\nB^\bullet\right)^{W(\mathfrak{s})}.\nonumber
\end{eqnarray}
Observe in the previous expression that terms $\nT$ and $\nT^\bullet$
appear divided by $u$: blocks on the dual map are considered in the term
$u^{\left|v_1(\mathfrak{s})\right|}$, so we do not consider the root of
the different non-crossing trees.

To obtain the asymptotic behavior in terms of the number of danglings, we
write $u=1$ in Equation~\eqref{eq:sing-u}. To study the resulting GF, we
need the expression of each factor of Equation~\eqref{eq:sing-u} when we
write $u=1$; all these expressions are shown in Table~\ref{table:trees}.
This table is built from the expressions for $\nT$ and $\nB$ deduced in
Lemma~\ref{lemma:planar-trees} and the expressions for double trees in
Table~\ref{table:double-trees}.

{\renewcommand{\arraystretch}{1.2}
\begin{table}[htb]
\begin{center}
\begin{tabular}{c|l}

  \textbf{GF} & $\,\,\,\,\,\,\,\,\,\,\,\,\,\,\,\,\,\,\,\,\,\,\,\,\,\,\,\,\,\,\,\,\,\,\,\,\,\,\,\,\,\,\,\,\,\,\,\,\,\,\,\,\,\,\,\,\,\,$\textbf{Expression}  \\\hline
  $\nT_1(z,1)$ & $1/16(1-4z)^{-1/2}-1/8(1-4z)^{1/2}+1/16(1-4z)^{3/2}$  \\
  $\nT_2(z,1)$ & $1/4(1-4z)^{-1/2}+1/2+(1-4z)^{1/2}$ \\
  $\nT_3(z,1)$ & $z^2\left(1/16(1-4z)^{-1/2}-1/8(1-4z)^{1/2}+1/16(1-4z)^{3/2}\right)$  \\
  $\nT(z,1)$ & $(1-(1-4z)^{1/2})/(2z)$  \\
  $\nB(z,1)$ & $\left(1-(1-4z)^{1/2}\right)/2$  \\
  $\nT^\bullet(z,1)$ & $(1-4z)^{-1/2}/z-(1-(1-4z)^{-1/2})/(2z^2)$  \\
  $\nB^\bullet(z,1)$ & $(1-4z)^{-1/2}$  \\
\end{tabular}\bigskip
\caption{Univariate GF for all families of trees.\label{table:trees}}
\end{center}
\end{table}}

The GF in Equation~\eqref{eq:sing-u} is a finite sum (a total of
$|\mathfrak{S}_\Sigma|$ terms), so its singularity is located at $z=1/4$
(since each addend has a singularity at this point). For each choice of
$\mathfrak{s}$,
\begin{equation}\label{eq:vertices}
\nT(z,1)^{\sum_{x\in v_1(\mathfrak{s})}(\dd(x)-2r(x))}\nB(z,1)^{\sum_{y\in
v_2(\mathfrak{s})}(\dd(y)-2r(y))}=\sum_{n=0}^{f(\mathfrak{s})}
f_n(z)(1-4z)^{n/2},
\end{equation}
where the positive integer $f(\mathfrak{s})$ depends only on
$\mathfrak{s}$, $f_n(z)$ are functions analytic at $z=1/4$, and
$f_0(z)\neq 0$ at $z=1/4$. For the other multiplicative terms, we obtain
\begin{equation}\label{eq:edges}
\nT_{1}(z,1)^{e_1(\mathfrak{s})}\nT_{2}(z,1)^{e_2(\mathfrak{s})}\nT_{3}(z,1)^{e_3(\mathfrak{s})}\nT^{\bullet}(z,1)^{B(\mathfrak{s})}\nB^\bullet(z,1)^{W(\mathfrak{s})}=G_{\mathfrak{s}}(z)(1-4z)^{-\frac{e(\mathfrak{s})}{2}}+\dots,
\end{equation}
where $G_{\mathfrak{s}}(z)$ is a function analytic at $z=1/4$. The reason
for this fact is that each factor in Equation~\eqref{eq:edges} can be
written in the form $p(z)(1-4z)^{-1/2}+\dots$, where $p(z)$ is a function
analytic at $z=1/4$, and
$e_1(\mathfrak{s})+e_2(\mathfrak{s})+e_3(\mathfrak{s})+B(\mathfrak{s})+W(\mathfrak{s})$
is the total number of edges. Multiplying Expressions~\eqref{eq:vertices}
and~\eqref{eq:edges} we obtain the contribution of a map $\mathfrak{s}$ in
$\nP_{\Sigma}(z,1)$. More concretely, the contribution of a single map
$\mathfrak{s}$ to Equation~(\ref{eq:sing-u}) can be written in the form
$$g_\mathfrak{s}(z)(1-4z)^{-e(S)/2}+\dots,$$
where $g_\mathfrak{s}(z)$ is a function analytic at $z=1/4$. Looking
at~\eqref{eq:transfer} from Proposition~\ref{thm:transfer}, we deduce that
the maps giving the greatest contribution to the asymptotic estimate of
$p_k^{\Sigma}$ are the ones maximizing the value $e(\mathfrak{s})$.
Applying Euler's formula (recall that all maps in $\mathfrak{S}_\Sigma$
have $\beta(\Sigma)$ faces) on $\overline{\Sigma}$ gives that these maps
are precisely the maps in $\mathfrak{C}_{\Sigma}$. In particular, maps in
$\mathfrak{C}_{\Sigma}$ have $2\beta(\Sigma)-3\chi(\Sigma)$ edges. Hence,
the singular expansion of $\nP_\Sigma(z,1)$ at $z=1/4$ is
\begin{equation}
\nP_\Sigma(z,1) \sim_{z\rightarrow 1/4}
c(\Sigma)(1-4z)^{3\chi(\Sigma)/2-\beta(\Sigma)}\left(1+O((1-4z)^{1/2})\right),
\end{equation}
where $c(\Sigma)=\sum_{\mathfrak{s}\in
\mathfrak{C}_{\Sigma}}g_\mathfrak{s}(1/4)$. Applying
Proposition~\ref{thm:transfer} on this expression yields the result as
claimed.
\end{proof}

\subsection{Irreducibility vs reducibility}\label{apen:4-enum}

For conciseness, in this subsection we write
$$a(\Sigma)=\frac{c(\Sigma)}{\Gamma\left(-3\chi(\Sigma)/2+\beta(\Sigma\right))}$$
to denote the constant term which appears in
Equation~\eqref{eq:asympt-enumeration-2} from
Lemma~\ref{thm:assympt-enumeration-1}. By
Lemma~\ref{lemma:irreducibility}, for a non-irreducible bipartite
subdivision $M$ of $\mathcal{R}_\Sigma$, there is a non-contractible cycle
$\mathbb{S}^1$ contained in a white two-dimensional region of $R$.
Additionally, $M$ induces a regular bipartite subdivision on the surface
$\Sigma\backslash\mathbb{S}^1=\Sigma'$, which can be irreducible or not.
By Lemma~\ref{lemma:regularitat}, each element of $\mathcal{R}_{\Sigma'}$
defines an element of $\mathcal{R}_{\Sigma}$. To prove that irreducible
bipartite subdivisions over $\Sigma$ give the maximal contribution to the
asymptotic, we apply a double induction argument on the pair
$(\chi(\Sigma),\beta(\Sigma))$. The critical point is the initial step,
which corresponds to the case where $\overline{\Sigma}$ is the sphere. The
details are shown in the following lemma.

\begin{lemma}\label{prop:genus-0-induction}
Let $\Sigma$ be a surface obtained from the sphere by deleting $\beta$
disjoints disks. Then
$$|\mathcal{R}_\Sigma(n)\backslash \mathcal{P}_{\Sigma}(n) |=o\left(|\mathcal{P}_{\Sigma}(n)|\right).$$
\end{lemma}
\begin{proof} We proceed by induction on $\beta$. The case $\beta=1$ corresponds
to a disk. We deduced in Subsection~\ref{apen:3-enum} the exact expression
for $\nP_{\Sigma}(z,u)$ (see Equation (\ref{eq:gen-catalan-function})). In
this case the equality $|\mathcal{R}_\Sigma(n)|=|\mathcal{P}_{\Sigma}(n)|$
holds for every value of $n$. Let us consider now the case $\beta=2$,
which corresponds to a cylinder. From
Equation~\eqref{eq:asympt-enumeration-2}, the number of irreducible
bipartite subdivisions on a cylinder verifies
\begin{equation}\label{eq:asint-cilindre}
\left|\mathcal{P}_\Sigma(n)\right|=_{n\rightarrow \infty} a(\Sigma)\cdot
n\cdot 4^{n}\left(1+O\left(n^{-1/2}\right)\right).
\end{equation}
Let us calculate upper bounds for the number of non-irreducible bipartite
subdivisions on a cylinder. A non-contractible cycle $\mathbb{S}^1$ on a
cylinder induces a pair of non-crossing partitions on a disk (one for each
boundary component of this cylinder). The asymptotic in this case is of
the form $[z^n]\nT(z,1)^2=_{n\rightarrow \infty} O(n^{-3/2} 4^n)$. The
subexponential term in Equation~\eqref{eq:asint-cilindre} is greater, so
the claim holds for $\beta=1$.

Let us proceed with the inductive step. Let $\beta>1$ be the number of
cycles in the boundary of $\Sigma$. A non-contractible cycle
$\mathbb{S}^1$ always separates $\Sigma$ into two connected components,
namely $\Sigma_1$ and $\Sigma_2$. By induction hypothesis,
$$|\mathcal{R}_{\Sigma_j}(n)\backslash \mathcal{P}_{\Sigma_j}(n)|=o\left(\left|\mathcal{P}_{\Sigma_j}(n)\right|\right),$$
for $j=1,2$. Consequently, we only need to deal with irreducible
decompositions of $\Sigma_1$ and $\Sigma_2$. The GF of regular bipartite
subdivisions that reduce to decompositions over $\Sigma_1$ and $\Sigma_2$
has the same asymptotic as $\nP_{\Sigma_1}(z,1)\cdot \nP_{\Sigma_2}(z,1)$.
The estimate of its coefficients is
$$[z^n]\nP_{\Sigma_1}(z,1)\nP_{\Sigma_2}(z,1)= a(\Sigma_1)a(\Sigma_2)[z^n](1-4z)^{-5/2\beta(\Sigma_1)+3}(1-4z)^{-5/2\beta(\Sigma_2)+3}. $$
Applying Proposition~\ref{thm:transfer} gives the estimate
$[z^n]\nP_{\Sigma_1}(z,1)\nP_{\Sigma_2}(z,1)=O\left(n^{5/2\beta-7}\cdot
4^n\right)$. Consequently, when $n$ is large enough the above term is
smaller than $p_{n}^{\Sigma}=O\left(n^{5/2\beta-4}4^n\right)$, and the
result follows.
\end{proof}

The next step consists in adapting the previous argument to surfaces of
positive genus. This second step is done in the following lemma.

\begin{lemma}\label{prop: irreducible-vs-reducible}
Let $\Sigma$ be a surface with boundary. Then
$$|\mathcal{R}_\Sigma(n)\backslash \mathcal{P}_{\Sigma}(n) |=o\left(\left|\mathcal{P}_{\Sigma}(n)\right|\right).$$
\end{lemma}

\begin{proof}
Let $\Sigma$ be a surface with boundary and Euler characteristic
$\chi(\Sigma)$. Consider a non-contractible cycle $\mathbb{S}^1$ contained
on a two-dimensional region. Observe that $\mathbb{S}^1$ can be either
one- or two-sided. Let $\Upsilon$ be the surface obtained from
$\Sigma\backslash\mathbb{S}^1$ by pasting a disk (or two disks) along the
cut (depending on whether $\mathbb{S}^1$ is one- or two-sided). Two
situations may occur:

\begin{enumerate}
\item $\Upsilon$ is connected and $\beta(\Upsilon)=\beta(\Sigma)$. In this case,
the Euler characteristic has been increased by either one if the cycle is
one-sided or by two if the cycle is two-sided. This result appears as
Lemma 4.2.4 in~\cite{Mohar:graphs-on-surfaces}.
\item The resulting surface  has two connected components, namely $\Upsilon=\Upsilon_1 \sqcup \Upsilon_2$. In this case, the total number of boundaries is $\beta(\Upsilon)=\beta(\Upsilon_1)+\beta(\Upsilon_2)$. By Lemma~\ref{lemma:particio}, $\chi(\Sigma)=\chi(\Upsilon_1) + \chi(\Upsilon_2)-2$.
\end{enumerate}

Clearly, the  base of the induction is given by
Lemma~\ref{prop:genus-0-induction}. The
 induction argument distinguishes between the following two cases: \\

\noindent{\em Case 1}. $\Upsilon$ is connected, by induction on the genus,
$|\mathcal{R}_{\Upsilon}(n)\backslash\mathcal{P}_{\Upsilon}(n)|<
_{n\rightarrow \infty}|\mathcal{P}_{\Upsilon}(n)|$. Additionally, by
Expression~\eqref{eq:asympt-enumeration-2}, an upper bound for
$|\mathcal{P}_{\Upsilon}(n)|$ is
$$[z^n]\nP_{\Upsilon}(z,1)= a(\Upsilon) n^{-3/2\chi(\Upsilon)+\beta(\Upsilon)-1} 4^n \left(1+O\left(n^{-1/2}\right)\right)= o\left( n^{-3/2\chi(\Sigma)+\beta(\Sigma)-1} 4^n \right). $$

\noindent{\em Case 2}. $\Upsilon$ is not connected. Then
$\Upsilon=\Upsilon_1 \sqcup \Upsilon_2$,
$\beta(\Sigma)=\beta(\Upsilon)=\beta(\Upsilon_1)+\beta(\Upsilon_2)$, and
$\chi(\Sigma)=\chi(\Upsilon_1) + \chi(\Upsilon_2)-2$. Again, by induction
hypothesis we only need to deal with the irreducible ones. Consequently,
$$[z^n]\nP_{\Upsilon_1}(z,1)\nP_{\Upsilon_2}(z,1)= a(\Upsilon_1)a(\Upsilon_2)[z^n](1-4z)^{3/2(\chi(\Upsilon_1)+\chi(\Upsilon_2))-(\beta(\Upsilon_1)+\beta(\Upsilon_2))+3}.$$
The exponent of  $(1-4z)$ in the last equation can be written as
$3/2\chi(\Sigma)-\beta(\Sigma)+3$. Consequently, the value
$[z^n]\nP_{\Upsilon_1}(z,1)\nP_{\Upsilon_2}(z,1)$ is bounded, for $n$
large enough, by
$$n^{-3/2\chi(\Sigma)+\beta(\Sigma)-3-1}\cdot 4^n = n^{-3/2\chi(\Sigma)+\beta(\Sigma)-4} 4^n=o\left(n^{-3/2\chi(\Sigma)+\beta(\Sigma)-1} 4^n\right).$$
Hence the contribution is smaller than the one given by
$|\mathcal{P}_{\Sigma}(n)|$, as claimed.
\end{proof}

\subsection{Upper bounds for non-crossing partitions}\label{subsection:final-result}

In this subsection we summarize all the steps in the previous subsections
of this section. Our main result is the following:
\begin{thm}\label{theorem: non-crossing-partition-final}
Let $\Sigma$ be a surface with boundary. Then the number
$|\Pi_{\Sigma}(n)|$ verifies
\begin{equation}\label{eq:thm-111}
|\Pi_{\Sigma}(n)|\leq _{n\rightarrow
\infty}\frac{c(\Sigma)}{\Gamma\left(-3/2\chi(\Sigma)+\beta(\Sigma)\right)}\cdot
n^{-3/2\chi(\Sigma)+\beta(\Sigma)-1} \cdot 4^n
\left(1+O\left(n^{-1/2}\right)\right),
\end{equation}
where $c(\Sigma)$ is a function depending only on $\Sigma$.
\end{thm}

\begin{proof}

By definition of non-crossing partition (recall
Subsection~\ref{subsection:definitions-non-crossing-partition})
$\left|\Pi_{\Sigma}(n)\right| \leq \left|\mmS_{\Sigma}(n)\right|$, as
non-crossing partitions are defined in terms of bipartite subdivisions,
and a different pair of bipartite subdivisions may define the same
non-crossing partition. We show in Lemma~\ref{lemma:only-regular} that in
fact $\left|\Pi_{\Sigma}(n)\right| \leq \left|\mR_{\Sigma}(n)\right|$, as
each bipartite subdivision can be reduced to a regular bipartite
subdivision by a series of joining boundaries and cutting vertices
operations. We partition the set $\mR_{\Sigma}(n)$ using the notion of
irreducibility (see Lemma~\ref{lemma:irreducibility}) in the form
$$\mR_{\Sigma}(n)= \mathcal{P}_{\Sigma}(n) \cup \left(\mathcal{R}_\Sigma(n)\backslash \mathcal{P}_{\Sigma}(n)\right).$$
Estimates for $\left|\mathcal{P}_{\Sigma}(n)\right|$ are obtained in
Lemma~\ref{thm:assympt-enumeration-1}, getting the bound stated in
Equation~\eqref{eq:asympt-enumeration-2}. In Lemma~\ref{prop:
irreducible-vs-reducible} we prove that $|\mathcal{R}_\Sigma(n)\backslash
\mathcal{P}_{\Sigma}(n)
|=o\left(\left|\mathcal{P}_{\Sigma}(n)\right|\right)$, hence the estimate
in Equation~\eqref{eq:thm-111} holds.\end{proof}

\section{Bounding $c(\Sigma)$ in terms of cubic maps}\label{apen:bounds-C}

In this section we obtain upper bounds for $c(\Sigma)$ by doing a more
refined analysis over functions $g_{\mathfrak{s}}(z)$ (recall the notation
used in Subsection~\ref{apen:3.5-enum}). This is done in the following
proposition.

\begin{lemma}\label{prop: C(Sigma)}
The function $c(\Sigma)$ defined in Lemma~\ref{thm:assympt-enumeration-1}
satisfies
\begin{equation}\label{eq: C(Sigma)}
c(\Sigma)\leq 2^{\beta(\Sigma)}|\mathfrak{C}_\Sigma|.
\end{equation}
\end{lemma}

\begin{proof} For each $\mathfrak{s}\in\mathfrak{C}_\Sigma$, we obtain bounds for $g_{\mathfrak{s}}(1/4)$. We use Table~\ref{table:trees2}, which is a
simplification of  Table~\ref{table:trees}. Now we are only concerned
about the constant term of each GF. Table~\ref{table:trees2} brings the
following information: the greatest contribution from double trees, trees,
and families of pointed trees comes from
$\mathcal{T}_{\square-\blacksquare}$, $\mathcal{T}$, and
$\mathcal{T}^{\bullet}$, respectively. The constants are $1/4$, $2$, and
$4$, respectively. Each cubic map has $-3\chi(\Sigma)+2\beta(\Sigma)$
edges ($\beta(\Sigma)$ of them being roots) and
$-2\chi(\Sigma)+\beta(\Sigma)$ vertices ($\beta(\Sigma)$ of them being
incident with roots). This characterization provides the following upper
bound for $g_S(1/4)$:
\begin{equation}\label{eq: C(Sigma2)}
g_{\mathfrak{s}}(1/4)\leq
\left(\frac{1}{4}\right)^{2\beta(\Sigma)-3\chi(\Sigma)-\beta(\Sigma)}
2^{-3\cdot
2\chi(\Sigma)+\beta(\Sigma)}4^{\beta(\Sigma)}=2^{\beta(\Sigma)}.
\end{equation}
\end{proof}

{\renewcommand{\arraystretch}{1.2}
\begin{table}[htb]
\begin{center}
\begin{tabular}{c|l|l}
  \textbf{GF} & $\,\,\,\,\,\,\,\,\,\,\,\,\,\,\,\,\,\,$\textbf{Expression}   & \textbf{Development at $z=1/4$}  \\\hline
  $\nT_1(z)$        & $(1-4z)^{-1/2}/16+\dots$          & $1/16(1-4z)^{-1/2}+\dots$  \\
  $\nT_2(z)$        & $(1-4z)^{-1/2}/4+\dots$           & $1/4(1-4z)^{-1/2}+\dots$ \\
  $\nT_3(z)$        & $z^2/16(1-4z)^{-1/2}+\dots$ & $1/256(1-4z)^{-1/2}+\dots$  \\
  $\nT(z)$          & $1/(2z)+\dots$                     & $2+\dots$  \\
  $\nB(z)$          & $1/2+\dots$                        & $1/2+\dots$  \\
  $\nT^\bullet(z)$  & $(1-4z)^{-1/2}/z+\dots$           & $4(1-4z)^{-1/2}+\dots$  \\
  $\nB^\bullet(z)$  & $(1-4z)^{-1/2}$                    & $(1-4z)^{-1/2}$  \\
\end{tabular}\bigskip
\caption{A simplification of Table~\ref{table:trees} used in
Lemma~\ref{thm:assympt-enumeration-1}.\label{table:trees2}}
\end{center}
\end{table}}

The value of $\mathfrak{C}_\Sigma$ can be bounded using the results
in~\cite{tg,cubic-maps}. Indeed, Gao shows in \cite{cubic-maps} that the
number of rooted cubic maps with $n$ vertices in an orientable surface of
genus\footnote{the \emph{genus} $g(\Sigma)$ of an orientable surface
$\Sigma$ is defined as $g(\Sigma) =1-\chi(\Sigma)/2$ (see
\cite{Mohar:graphs-on-surfaces}).} $g$ is asymptotically equal to
$$t_g \cdot n^{5(g-1)/2}\cdot (12\sqrt{3})^n,$$

where the constant $t_g$ tends to zero as $g$ tends to infinity~\cite{tg}.
A similar result is also stated in~\cite{cubic-maps} for non-orientable
surfaces. By duality, the number of rooted cubic maps on a surface
$\overline{\Sigma}$ of genus $g(\Sigma)$ with $\beta(\Sigma)$ faces is
asymptotically equal to $t_{g(\Sigma)} \cdot
\beta(\Sigma)^{5\left(g(\Sigma)-1\right)/2}\cdot
(12\sqrt{3})^{\beta(\Sigma)}$.


To conclude, we observe that the elements of $\mathfrak{C}_{\Sigma}$ are
obtained from rooted cubic maps with $\beta(\Sigma)$ faces by adding a
root on each face different from the root face. Observe that each edge is
incident with at most two faces, and that the total number of edges is
$-3\chi(\Sigma)$. Consequently, the number of ways of rooting a cubic map
with $\beta(\Sigma)-1$ unrooted faces is bounded by
$\binom{-6\chi(\Sigma)}{\beta(\Sigma)-1}$.


Lemma~\ref{prop: C(Sigma)}, together  with the discussion above, yields
the following bound for $c(\Sigma)$.

\begin{prop}\label{prop:bound_C-sigma}
The constant $c(\Sigma)$ verifies

$$c(\Sigma)< t_{1-\chi(\Sigma)/2} \cdot \beta(\Sigma)^{-5\chi(\Sigma)/2}\cdot (12\sqrt{3})^{\beta(\Sigma)}\cdot \binom{-6\chi(\Sigma)}{\beta(\Sigma)-1}\cdot 2^{\beta(\Sigma)}.
$$
%
\end{prop}

\SP
\paragraph{\textbf{Further research.}} In this article, we provided upper bounds for $|\Pi_{\Sigma}(n)|$. This upper bound is exact for the exponential growth (recall Section~\ref{sec:intro}). However,
we cannot assure exactness for the subexponential growth: the main problem
in order to state asymptotic equalities is that $|\Pi_{\Sigma}(n)| \neq
|\mathcal{P}_{\Sigma}(n)|$: there are different irreducible bipartite
subdivisions with $n$ vertices which define the same non-crossing
partition (see Figure~\ref{twodif} for an example). Hence, an open problem
in this context is finding more precise lower bounds for the number of
non-crossing partitions.


%
%

%


Another interesting problem is based on generalizing the notion of
$k$-triangulation to the partition framework and getting the asymptotic
enumeration: the enumeration of $k$-triangulations on a disk was found
using algebraic methods in~\cite{Jonsson-k-triangulations}. This notion
can be easily translated to the non-crossing partition framework on a
disk, and the {\sl exact} enumeration in this case seems to be more
involved. In the same way as non-crossing partitions on surfaces play a
crucial role for designing algorithms for graphs on surfaces
(see~\cite{RST10_algo_Arxiv}), it turns out that the enumeration mentioned
above is of capital importance in order to design algorithm for families
of graphs defined by excluding minors.

$ $\\
\paragraph{\textbf{Acknowledgements.}} We would like to thank Marc Noy for  pointing us to references~\cite{tg,cubic-maps}.


{\small \bibliography{comb}}
\bibliographystyle{acm}
\end{document}